\documentclass{amsart}

\usepackage{amssymb,amsfonts,amsmath}
\usepackage[mathscr]{eucal}

\makeindex

\newtheorem{thm}{{{Theorem}}}[section]

\newtheorem{lem}[thm]{{Lemma}}

\newtheorem{rem}[thm]{{Remark}}
\numberwithin{equation}{section}

\def\R{\mathbb{R}}
\def\C{\mathbb{C}}
\def\A{\mathbb{A}}

\def\GL{{\mathrm{GL}}}
\def\PGL{{\mathrm{PGL}}}
\def\SL{{\mathrm{SL}}}

\def\SO{{\mathrm{SO}}}

\def\Sp{{\mathrm{Sp}}}
\def\GSp{{\mathrm{GSp}}}

\def\bG{{\mathbb{G}}}

\def\Ad{{\mathrm{Ad}}} 
\def\vol{{\mathrm{vol}}}

\def\d{{\mathrm{d}}}

\def\bsl{\backslash}
\def\inf{\infty}

\def\bK{{\mathbf{K}}}

\def\fa{{\mathfrak{a}}}

\def\fg{{\mathfrak{g}}}
\def\fn{{\mathfrak{n}}}

\def\Ad{{\mathrm{Ad}}}
\def\La{{\langle}}
\def\Ra{{\rangle}}
\def\tJ{{\tilde{J}}}
\def\tZ{{\tilde{Z}}}

\begin{document}

\pagestyle{plain}

\title{The subregular unipotent contribution to the geometric side of the Arthur trace formula for the split exceptional group $G_2$}

\author{Tobias Finis}
\address{Mathematisches Institut, Universit\"at Leipzig, PF 100920 D-04009 Leipzig, Germany}
\email{Tobias.Finis@math.uni-leipzig.de}

\author{Werner Hoffmann}
\address{Fakult\"at f\"ur Mathematik, Universit\"at Bielefeld, PF 100131, D-33501 Bielefeld, Germany}
\email{hoffmann@math.uni-bielefeld.de}

\author{Satoshi Wakatsuki}
\address{Faculty of Mathematics and Physics, Institute of Science and Engineering, Kanazawa University, Kakumamachi, Kanazawa, Ishikawa, 920-1192, Japan}
\email{wakatsuk@staff.kanazawa-u.ac.jp}

\thanks{The second author is partially supported by the Collaborative Research Center 701 of the DFG. The third author is partially supported by JSPS Grant-in-Aid for Scientific Research (No. 26800006, 25247001, 15K04795).}

\maketitle

\begin{abstract}
In this paper, a zeta integral for the space of binary cubic forms is associated with the subregular unipotent contribution to the geometric side of the Arthur trace formula for the split exceptional group $G_2$.
\end{abstract}

\tableofcontents

\section{Introduction}

On the geometric side of the Arthur trace formula, the properties of global coefficients are unknown in general, but they should be explained by zeta functions of prehomogeneous vector spaces.
A crucial step for such an investigation is to relate the contribution of each geometric conjugacy class to a zeta integral of a corresponding prehomogeneous vector space.
In this paper, we perform this task for the split exceptional group of type~$G_2$ over any algebraic number field.
We treat only the subregular unipotent contribution to the trace formula.
Our main result (Theorem \ref{main}) relates it to a zeta integral for the space of binary cubic forms.
The other unipotent contributions behave in a familiar way (see Remark \ref{rem1}).

Shintani was the first to introduce the zeta integral and zeta functions for the space of binary cubic forms (cf. \cite{Shintani}).
He showed their meromorphic continuation by using Eisenstein series (see \cite{Wright} for the adelization).
Actually, our study is closely related to Shintani's work, but his method is not suitable for the modified kernels of the trace formula.
We rather use Kogiso's method \cite{Kogiso} in the proof, because it is simple and does not require Eisenstein series.
Our argument follows the lines of the general direction \cite{Hoffmann} and includes proofs of some of its conjectures in case of~$G_2$.

For earlier work on unipotent contributions and global coefficients, we refer to \cite{Ch1,Ch2,CL,Matz,Matz2} for $\GL(n)$, to \cite{HW} for $\GSp(2)$ and $\Sp(2)$ (rank two), and to \cite{Hoffmann1} for the rank one case (non-adelic).
Information about properties of global coefficients has several applications.
For example, it can be used to study the asymptotic behaviour of Hecke eigenvalues (see \cite{Matz3,MT,KWY}).

\section{Main result}\label{maintheorem}

In this section, we present our main result.
Let us explain notations.
We write $F$ for an algebraic number field and $\A$ for the adele ring of $F$.
Fix a non-tivial additive character $\psi_F$ of $F\bsl \A$.
The measure on $\A$ with $\vol(F\bsl \A)=1$ is self-dual for~$\psi_F$.
Let $|\;\;|$ denote the idele norm on the idele group $\A^\times$ and set $\A^1=\{a\in\A^\times \mid |a|=1\}$.

A split simple algebraic group $G$ of type~$G_2$ over $F$ is defined as the automorphism group of the split octonion algebra over $F$. It is connected and can be realized as a closed subgroup of the split special orthogonal group $\SO(7)$ over $F$.

For every Levi subgroup $M$ of $G$ over $F$,
we write $X(M)_F$ for the abelian group of $F$-rational characters on $M$.
We set $\fa_M=\mathrm{Hom}(X(M)_F,\R)$ and $\fa_M^*=X(M)_F\otimes \R$.
A mapping $H_M:M(\A)\to\fa_M$ is defined by $\langle H_M(m),\chi\rangle=\log|\chi(m)|$ for $m\in M(\A)$, $\chi\in X(M)_F$.
Let $M(\A)^1$ denote the kernel of $H_M$ and let $A_M$ denote the $F$-split part of the center of $M$.

We choose a minimal parabolic subgroup $P_0$ over $F$ and a Levi component $M_0$ of $P_0$ over $F$.
The unipotent radical of $P_0$ is denoted by $N_0$.
There is a maximal compact subgroup $\bK$ of $G(\A)$ which is admissible relative to $M_0$.

Since the rank of $G$ is two, we have two maximal parabolic subgroups $P_1$, $P_2$ containing~$P_0$.
Let $M_j$ denote the Levi subgroup of $P_j$ that contains $M_0$ and let $N_j$ denote the unipotent radical of $P_j$.
For each $P_j$ $(j=0,1,2)$, we define a mapping $H_{P_j}:G(\A)\to \fa_{M_j}$ by $H_{P_j}(nmk)=H_{M_j}(m)$ for $n\in N_j(\A)$, $m\in M_j(\A)$, and $k\in\bK$.

In the set of positive roots with respect to~$P_0$, we have the subset $\Delta_0=\{\alpha_1,\alpha_2\}$ of simple roots, where $\alpha_1$ and $\alpha_2$ are determined by
\[
\alpha_1|_{A_{M_2}}=1 \quad \text{and} \quad \alpha_2|_{A_{M_1}}=1.
\]
We choose the numbering in such a way that $\alpha_1$ is short and $\alpha_2$ is long.
We also have the set $\Delta_0^\vee=\{\alpha_1^\vee,\alpha_2^\vee\}\subset \fa_{M_0}$ of simple coroots and the set $\widehat\Delta_0=\{\varpi_1,\varpi_2\}$ of simple weights corresponding to~$\Delta_0$, which satisfy $\langle \alpha_j^\vee,\varpi_k\rangle=\delta_{jk}$.
The corresponding sets of simple weights for $P_1$ and~$P_2$ are $\widehat\Delta_1=\{\varpi_1\}$ and $\widehat\Delta_2=\{\varpi_2\}$.
For each $P_j$ $(j=0,1,2)$, we denote by $\widehat\tau_{P_j}$ the characteristic function on $\fa_{M_j}$ of the set
\[
\{H\in\fa_{M_j} \mid \varpi(H)>0 \;\; (\forall \varpi \in \widehat\Delta_j)\}.
\]

Fix a Haar measure $\d g$ on $G(\A)$ and normalize the Haar measures $\d k$ on $\bK$ and on $P_j(\A)\cap\bK$ by $\vol(\bK)=\vol(P_j(\A)\cap\bK)=1$. There is a unique left Haar measure on $P_j(\A)$ such that the isomorphism $P_j(\A)/P_j(\A)\cap\bK\to G(\A)/\bK$ of $P_j(\A)$-spaces preserves the invariant measure. We normalize the measure $\d n$ on $N_j(\A)$ by $\vol(N_j(F)\bsl N_j(\A))=1$, so that a Haar measure $\d m$ on $M_j(\A)$  is now determined. We endow $\fa_{M_j}$ with the measure such that the quotient by the lattice spanned by the (projected) simple coroots has volume~1. This fixes a measure $\d^1 m$ on $M_j(\A)^1$, so that the volumes
\[
\vol_{M_j}=\vol(M_j(F)\bsl M_j(\A)^1)\quad (j=0,1,2)
\]
are now determined.

Let $C$ denote the geometric subregular unipotent conjugacy class of $G$.
It is known that the dimension of $C$ is $10$ (see \cite{CM}). Following \cite[Section 1.3]{Hoffmann}, we write $Q$ for a canonical parabolic subgroup (also called a Jacobson-Morozov parabolic subgroup) of an element of $C(F)$, $U$ for its unipotent radical, and $L$ for its Levi subgroup containing $M_0$.
This means that $Q=P_2$, $U=N_2$, and $L=M_2$.
The object of our investigation is the subregular unipotent contribution $J_C^T(f)$ to the geometric side of the trace formula, which is defined as
\begin{multline}\label{cont}
J_C^T(f)= \int_{G(F)\bsl G(\A)} \Big\{ \sum_{\gamma\in C(F)}f(g^{-1}\gamma g)\\
- \sum_{j=1}^2 \sum_{\delta\in P_j(F)\bsl G(F)}\int_{N_j(\A)} f(g^{-1}\delta^{-1}n \delta g)\, \d n \, \widehat\tau_{P_j}(H_{P_j}(\delta g)-T) \Big\}\d g,
\end{multline}
where $f\in C_c^\inf(G(\A))$ is a test function and $T\in\fa_{M_0}$ is a truncation parameter.
It is known that the integral $J_C^T(f)$ is absolutely convergent by \cite[Theorem 7.1]{FL} (see also \cite{Ch2}).
Using \cite[Theorem 4.2 and Corollary 8.4]{Arthur1}, one can  express it as a linear combination of local unipotent weighted orbital integrals.
The coefficients in the linear combination are called global coefficients.

In the Lie algebra $\fg$ of $G$, a $4$-dimensional $F$-vector space $V$ is defined as the direct sum of root spaces of $\alpha_2$, $\alpha_1+\alpha_2$, $2\alpha_1+\alpha_2$, and $3\alpha_1+\alpha_2$.
We write $x\cdot l$ $(x\in V$, $l\in L)$ for the right action of $L$ on $V$ over $F$, which is defined by $\Ad(l^{-1})$.
Then, the pair $(L,V)$ is a prehomogeneous vector space over $F$ and can be identified with the space of binary cubic forms (see Section \ref{g2} for details).
The set $V$ can be identified with $U/U'$ via the exponential mapping, where $U'=[U,U]$ denotes the derived subgroup of $U$.
Similarly, the sum $V_1$ of the root spaces of $\alpha_1$ and $\alpha_1+\alpha_2$ is identified with $N_1/N_1'$ by the exponential mapping, where $N_j'=[N_j,N_j]$.
We normalize measures on $N_j'(\A)$ by $\vol(N_j'(F)\bsl N_j'(\A))=1$.

For the test function $f$ on $G(\A)$, we set
\[
f_{\bK,U'}(x)=\int_\bK \int_{U'(\A)}f(k^{-1}\exp(x) u'k)\, \d u' \, \d k \quad (x\in V(\A)),
\]
\[
f_{\bK,N_1'}(y)=\int_\bK \int_{N_1'(\A)}f(k^{-1}\exp(y)n'k)\, \d n' \, \d k \quad (y\in V_1(\A)),
\]
\[
f_{\bK,N_j}(1)=\int_\bK \int_{N_j(\A)} f(k^{-1}nk)\, \d n \, \d k \quad \text{$(j=1$, $2)$}.
\]

A zeta integral is defined by
\begin{equation}\label{zeta}
Z(\varphi,s)=\int_{L(F)\bsl L(\A)} e^{-s\langle\varpi_2,H_Q(l)\rangle} \sum_{x\in V^0(F)} \varphi(x\cdot l)\, \d l \quad (s\in\C),
\end{equation}
where $\varphi$ is a Schwartz-Bruhat function on $V(\A)$ and $V^0$ denotes the regular $L$-orbit  in $V$.
It is known that $Z(\varphi,s)$ is absolutely convergent if $\mathrm{Re}(s)>2$ (cf. \cite{Shintani,Wright,Saito2}).
Furthermore, it was proved that $Z(\varphi,s)$ can be meromorphically continued to the whole $s$-plane and has at most simple poles at $s=0$, $1/3$, $5/3$, and $2$ (see \eqref{prin}).
Here is our main result.

\begin{thm}\label{main}
For any $f\in C_c^\inf(G(\A))$ and any $T=T_1\alpha_1^\vee+T_2\alpha_2^\vee$, we have
\begin{align*}
J_C^T(f)=& \lim_{s\to 2+0} \frac{\d}{\d s}(s-2) Z(f_{\bK,U'},s) -\vol_{M_1}\int_{V_1(\A)} f_{\bK,N_1'}(y) \,   \log\|y\| \, \d y \\
&+ \sum_{j=1}^2  T_j \,  \vol_{M_j}f_{\bK,N_j}(1),
\end{align*}
where $\|\,\,\|$ is a suitably normalized $\bK\cap P_1(\A)$-invariant height function on~$V_1(\A)$.
\end{thm}

The proof will be given in Section~\ref{proof}.
Let us explain the relation between Theorem \ref{main} and global coefficients.
Fix a finite set $S$ of places of $F$ including all infinite places and set $F_S=\prod_{v\in S}F_v$ where $F_v$ denotes the completion of $F$ at $v$.
Assume that $S$ is sufficiently large.
The integral of $f_{\bK,N_1'}$ in the formula is essentially the derivative of the Tate integral at $2$.
Hence, it is expressed by the derivative of the product of a local zeta integral over $F_S$ and the Dedekind zeta function outside~$S$.
Using \cite{DW} or \cite{Saito1}, one can also express $Z(f_{\bK,U'},s)$ as a finite sum of products of local zeta integrals over $F_S$ and zeta functions outside~$S$.
Thus, the global coefficients can be expressed by such zeta functions.
However, in order to determine them precisely, one has to compute Arthur's weight factors as in~\cite[Section 5]{HW}) in order to compare them with the weight factors in derivatives of local zeta integrals over $F_S$.

\begin{rem}\label{rem1}\normalfont
The group $G$ has the five geometric unipotent conjugacy classes (see \cite{CM}).
There are three rigid unipotent orbits of dimensions $0$, $6$, and~$8$.
For each rigid class $O$, the contribution $J_O^T(f)=J_O(f)$ can be easily studied, because it need not be truncated, that is,
\[
\int_{G(F)\bsl G(\A)} \sum_{\gamma\in O(F)} |f(x^{-1}\gamma x)|\, \d x<\inf .
\]
The contribution of the unit element $(\dim=0)$ equals $\vol(G(F)\bsl G(\A))\, f(1)$.
The minimal unipotent contribution $(\dim=6)$ is expressed by a product of $\vol_{M_2}$ and the Tate integral over $F$ at $s=3$.
The contribution of the other rigid class $(\dim=8)$ is expressed as a product of $\vol_{M_1}$ and the Tate integral over $F$ at $s=2$.
In terms of the study \cite{DK} of local stable distributions, we guess that these values correspond to the contributions of unit elements of the endoscopic groups $\PGL(3)$ and $\SL(2)\times \SL(2)/\mu_2$ of $G$ (cf.~\cite{GG}).

Beside the rigid orbits and the subregular orbit $(\dim=10)$ already introduced, the remaining case is the regular unipotent orbit $(\dim=12)$.   
Its contribution to the trace formula is related to the Tate integral for $F\oplus F$ and can be studied by an argument similar to \cite{Ch2}, \cite{Matz} or \cite{HW}.
We omit its discussion since it is lengthy and presents no novelty.
\end{rem}

\begin{rem}\label{rem2}\normalfont
In \cite[Section 3.3]{Hoffmann}, the second author indicated that the $F$-rational points in a geometric conjugacy class should be partitioned into finer classes, called truncation classes, depending on which parabolic subgroups intervene in their truncation.
In the present situation, geometric orbits in $V^0(F)$ are divided into three classes related to field extensions $E$ of~$F$.
In \cite[Section 8]{Taniguchi}, Taniguchi decomposed the zeta integral as $Z(\varphi,s)=\sum_{i=1}^3Z_{(i)}(\varphi,s)$ according to the index $i=[E:F]$.
His result \cite[Proposition 8.6]{Taniguchi} implies that $Z_{(1)}(f_{\bK,U'},2)$ is convergent, while $Z_{(i+1)}(f_{\bK,U'},s)$ should be related to the truncation $\widehat\tau_{P_i}(H_{P_i}(g)-T)$ $(i=1,2)$.
We do not study the contribution of each truncation class in this paper, but it would be interesting to understand this phenomenon.
\end{rem}

\section{The group of type $G_2$ and the space of binary cubic forms}\label{g2}

First, we recall some known facts on the structure of $G$.
For details, we refer to \cite{BS,CNP,GGS,SV}.

The minimal Levi subgroup $M_0$ is a maximal split torus in $G$ over~$F$, and $\Delta_0$  is a basis of the abelian group~$X(M_0)$. The system $\Phi_+\subset X(M_0)$ of positive roots corresponding to $P_0$ is given by
\[
\Phi_+=\{\alpha_1,\;\; \alpha_2, \;\; \alpha_1+\alpha_2, \;\; 2\alpha_1+\alpha_2, \;\; 3\alpha_1+\alpha_2, \;\; 3\alpha_1+2\alpha_2\}.
\]
We set $\Phi_-=-\Phi_+$ and $\Phi=\Phi_+\sqcup\Phi_-$.
We have a Chevalley basis $\{H_{\alpha_1},H_{\alpha_2},X_\alpha\mid \alpha\in\Phi\}$ of $\fg$ (cf. \cite{Steinberg} and \cite[p.293]{CNP}).
Now, the set $\{X_\alpha\mid \alpha\in\Phi_+\setminus \{ \alpha_{3-j} \}\}$ ($j=1$ or $2$) is a basis of the $F$-vector space $\fn_j=\mathrm{Lie}(N_j)$.

In this setting, it follows that
\[
\varpi_1=2\alpha_1+\alpha_2 , \quad \varpi_2=3\alpha_1+2\alpha_2.
\]
We choose a new basis $\{H_1,H_2\}$ in $\fa_0$ by
\[
\alpha_1^\vee=H_2-H_1 ,\quad \alpha_2^\vee=H_1 .
\]
Then, it follows that 
\[
\alpha_1(t_1H_1+t_2H_2)=t_2-t_1,\quad \alpha_2(t_1H_1+t_2H_2)=2t_1-t_2.
\]
for $t_1,t_2\in\R$.
Our basis $\{H_1,H_2\}$ corresponds to the parametrization $M_0\cong \bG_m\times \bG_m$ for which
\[
\alpha_1(a,b)=ba^{-1},\quad \alpha_2(a,b)=a^2b^{-1} \qquad ((a,b)\in M_0),
\]
and we find that $M_1$ and $L=M_2$ are isomorphic to $\GL(2)$ over $F$.
For $m=(a,b)\in M_0(\A)$, one has
\[
H_0(m)=(\log|a|)H_1+(\log|b|)H_2\in\fa_0,
\]
and it follows that
\[
\widehat\tau_{P_1}(H_{P_1}(m)-T)=1 \;\; \Leftrightarrow \;\; |b|>e^{T_1}, \quad  \widehat\tau_{P_2}(H_{P_2}(m)-T)=1 \;\; \Leftrightarrow \;\; |ab|>e^{T_2} .
\]

Next, we relate a subspace of $\fn_2$ to the space of binary cubic forms.
By the following identifications
\[
u^3=X_{\alpha_2}, \;\; 3u^2v=X_{\alpha_1+\alpha_2} , \;\;  3uv^2=X_{2\alpha_1+\alpha_2} , \;\; v^3= X_{3\alpha_1+\alpha_2},
\]
where $u$ and $v$ are variables,  the $F$-vector space
\[
V=\La X_{\alpha_2} , X_{\alpha_1+\alpha_2} , X_{2\alpha_1+\alpha_2} , X_{3\alpha_1+\alpha_2} \Ra (\subset \fg)
\]
is identified with the space of binary cubic forms over $F$.
The group $L=\GL(2)$ acts on $V$ by
\[
f(u,v)\cdot l=\det(l)\, f((u,v)l^{-1}) \qquad (l\in L)
\]
for each binary cubic form $f(u,v)$ in $V$.
This is the restriction of the action $\mathrm{Ad}(l^{-1})$ on $\fg$ to $V$.
We identify $V(F)$ with $F^{\oplus 4}$ by the isomorphism
\[
x_1X_{\alpha_2}+\frac{x_2}{3}X_{\alpha_1+\alpha_2}+\frac{x_3}{3} X_{2\alpha_1+\alpha_2}+x_4X_{3\alpha_1+\alpha_2}\mapsto (x_1,x_2,x_3,x_4).
\]
For $x=(x_1,x_2,x_3,x_4)$ and $y=(y_1,y_2,y_3,y_4)$ in $V$, we set
\[
[x,y]=x_1y_4-\frac{1}{3}x_2y_3+\frac{1}{3}x_3y_2-x_4y_1.
\]
These notations are the same as in \cite{Wright}.
By this bilinear form, the dual space of $V$ is identified with $V$ itself in an $L$-equivariant fashion.
The discriminant $P(x)$ of $x=(x_1,x_2,x_3,x_4)\in V$ is given by
\[
P(x)=x_2^2x_3^2+18x_1x_2x_3x_4-4x_2^3x_4-4x_1x_3^3-27x_1^2x_4^2.
\]
Hence, the regular geometric $L$-orbit in $V$ is given by $V^0=\{x\in V \mid P(x)\neq 0\}$.

We already introduced the zeta integral $Z(\varphi,s)$ in \eqref{zeta}, where $s\in\C$ and $\varphi$ is a Schwartz-Bruhat function on $V(\A)$.
We may assume without loss of generality that $\varphi(x\cdot k)=\varphi(x)$ holds for any $k\in\bK$ and $x\in V(\A)$.
The Haar measure $\d x$ on $V(\A)$ is normalized by $\vol(V(\A)/V(F))=1$. We also choose a Haar measure $\d^\times a$ on $\A^\times$. Together with the measure $\d t/t$ on~$\R_{>0}$, this determines a measure $\d^1a$ on~$\A^1$, and we set
\[
 c_F=\vol(F^\times \bsl \A^1).
\]
Consider the partial zeta integral
\[
Z_+(\varphi,s)=\int_{L(F)\bsl L(\A),\, |\det(l)|<1} |\det(l)|^{-s} \sum_{x\in V^0(F)} \varphi(x\cdot l)\, \d l
\]
and the Fourier transform resp. singular orbital integral
\[
\hat\varphi(y)=\int_{V(\A)} \varphi(x) \, \psi_F([x,y])\, \d x, \qquad \Sigma_1(\varphi,s)=\int_{\A^\times}|a|^s \, \varphi(0,0,0,a) \, \d^\times a.
\]
If $\mathrm{Re}(s)>2$, then one has
\begin{multline}\label{prin}
Z(\varphi,s)=  Z_+(\varphi,s)+Z_+(\hat\varphi,2-s) -\frac{1}{s} \, \vol_L \, \varphi(0) + \frac{1}{s-2} \, \vol_L \, \hat\varphi(0) \\
 - \frac{\vol_{M_0}}{c_F(3s-1)}\, \Sigma_1(\varphi,2/3) + \frac{\vol_{M_0}}{c_F(3s-5)}\, \Sigma_1(\hat\varphi,2/3) \\
 - \frac{\vol_{M_0}}{c_F \, s} \int_{\A^\times}\int_\A \int_\A |a|^2 \hat\varphi(0,a,x_3,x_4) \, \d x_3\, \d x_4 \, \d^\times a \\
 + \frac{\vol_{M_0}}{c_F(s-2)} \int_{\A^\times}\int_\A \int_\A |a|^2 \varphi(0,a,x_3,x_4) \, \d x_3 \, \d x_4 \, \d^\times a  .
\end{multline}
This can be deduced from the results of \cite{Shintani,Wright,Kogiso} by an argument which will reappear in the proof of Lemma~\ref{l4} below.
Since $Z_+(\varphi,s)$ is an entire function of~$s$, this provides the meromorphic continuation of $Z(\varphi,s)$ to the whole $s$-plane.

\section{Proof of the Main result}\label{proof}

In this section, we shall prove Theorem \ref{main}, our main result.
The subregular unipotent contribution $J_C^T(f)$ was already defined in \eqref{cont} for $f\in C_c^\inf(G(\A))$.
We define a modified version as
\begin{multline*}
\tJ_C^T(f)=\int_{G(F)\bsl G(\A)} \Big\{ \sum_{\gamma\in C(F)}  f(g^{-1}\gamma g)   \\
- \sum_{\delta\in Q(F)\bsl G(F)}\int_{U(\A)} f(g^{-1}\delta^{-1}u \delta g)\, \d u \, \widehat\tau_Q(H_Q(\delta g)-T)   \\
- \sum_{\delta\in P_1(F)\bsl G(F)}\sum_{\substack{ \nu\in N_1(F)/N_1'(F) \\ \nu\notin N_1'(F) }} 
\int_{N_1'(\A)} f(g^{-1}\delta^{-1} \nu n' \delta g)\, \d n' \, \widehat\tau_{P_1}(H_{P_1}(\delta g)-T)  \Big\}\d g.
\end{multline*}
\begin{lem}\label{change}
For each $T\in \fa_0^+$, the integral $\tJ_C^T(f)$ is absolutely convergent and we have 
\[
J_C^T(f)=\tJ_C^T(f).
\]
\end{lem}
\begin{proof}
Applying the Poisson summation formula to $V_1(F)= FX_{\alpha_1}+FX_{\alpha_1+\alpha_2}=\Ad(M_1(F))(F X_{\alpha_1+\alpha_2})$, one can prove
\begin{multline*}
\int_\bK \int_{P_1(F) \bsl P_1(\A)}\Big| \sum_{\delta\in P_0(F)\bsl P_1(F)}\sum_{x\in F^\times }  \int_{N_1'(\A)} f(k^{-1}p^{-1}\delta^{-1} \exp(x X_{\alpha_1+\alpha_2}) n' \delta pk)\, \d n' \\ 
-  \int_{N_1(\A)} f(k^{-1}p^{-1}n pk)\, \d n \Big|\, \widehat\tau_{P_1}(H_{P_1}(p)-T) \d p \, \d k
\end{multline*}
is convergent. 
Therefore, one sees that $\tJ_C^T(f)$ is absolutely convergent by considering the difference $\tJ_C^T(f)-J_C^T(f)$.
Furthermore, the equality is derived from the mean value formula
\begin{equation}\label{mean}
\int_{H(F)\bsl H(\A)} \sum_{y\in F\oplus F\setminus \{(0,0)\}} \phi(y\cdot h)\, \d h= \int_{H(F)\bsl H(\A)}\d h \;  \int_{\A\oplus\A} \phi(x) \, \d x,
\end{equation}
where $H=\SL(2)$, $\d h$ is a Haar measure on $H(\A)$, and $\phi$ is a Schwartz-Bruhat function on $\A\oplus \A$.
\end{proof}
In the above proof, we needed only a special case of the mean value formula,which was studied in more general situations by Siegel, Weil and Ono.

The following lemma will be proved in Section~\ref{prooflem1}.
\begin{lem}\label{tJconv}
The integral
\begin{multline*}
\int_{Q(F)\bsl G(\A)}\Big|   \sum_{\mu\in V^0(F)}\sum_{\nu\in U'(F)} f(g^{-1}\exp(\mu)\nu  g)  - \int_{U(\A)} f(g^{-1}u g)\, \d u \, \widehat\tau_Q(H_Q(g)-T)   \\
- \sum_{\delta\in P_0(F)\bsl Q(F)}\sum_{x\in F^\times }  \int_{N_1'(\A)} f(g^{-1}\delta^{-1} \exp(x X_{\alpha_1+\alpha_2}) n' \delta g)\, \d n' \, \widehat\tau_{P_1}(H_{P_1}(\delta g)-T)  \Big| \, \d g
\end{multline*}
is convergent
\end{lem}

We set
\begin{multline*}
\tZ(\varphi,s,T) = \int_{L(F)\bsl L(\A)} |\det(l)|^{-s} \Big\{  \sum_{\mu\in V^0(F)}\varphi(\mu\cdot l) - \int_{V(\A)} \varphi(u\cdot l)\, \d u  \, \widehat\tau_Q(H_Q(l)-T)   \\
 - \sum_{\delta\in B(F)\bsl L(F)}\sum_{x\in F^\times } \, \int_\A \int_\A \varphi((0,x,x_3,x_4)\cdot \delta l) \, \d x_3 \, \d x_4  \, \widehat\tau_{P_1}(H_{P_1}(\delta l)-T)  \Big\}\d l
\end{multline*}
where we put $B=L\cap P_0$.
Note that $B$ is the lower triangular subgroup of $L$ if we realize $L$ as $\GL(2)$ according to Section \ref{g2}.
The following lemma is proved in Section \ref{prooflem2}.
\begin{lem}\label{l4}
The integral $\tZ(\varphi,s,T)$ is absolutely convergent for $\mathrm{Re}(s)>5/3$.
For $\mathrm{Re}(s)>5/3$, we have
\begin{align*}
\tZ(\varphi,s,T&)=  Z_+(\varphi,s)+Z_+(\hat\varphi,2-s) - \frac{1}{s} \, \vol_L \, \varphi(0) + \int_1^{e^{T_2}} t^{-(s-2)}\, \d^\times t \, \vol_L \, \hat\varphi(0) \\
& - \frac{\vol_{M_0}}{c_F(3s-1)}\, \Sigma_1(\varphi,2/3) + \frac{\vol_{M_0}}{c_F(3s-5)}\, \Sigma_1(\hat\varphi,2/3) \\
& - \frac{\vol_{M_0}}{c_F \, s} \int_{\A^\times}\int_\A \int_\A |a|^2 \hat\varphi(0,a,x_3,x_4) \, \d x_3\, \d x_4 \, \d^\times a \\
& + \frac{\vol_{M_0}}{c_F^2} \, \int_{F^\times\bsl\A^\times} \int_{\A^\times}
\int_{\A^{\oplus 2}}  \varphi(0,a,x_3,x_4) \, \d x_3 \, \d x_4 \\
&    \quad \times |a|^{s} |b|^{2-s}(\hat\tau_Q(H_Q(m))-\hat\tau_{P_1}(H_{P_1}(m)-T))\,\d^\times a\, \d^\times b  
\end{align*}
where $m=(a^{-1},b)\in M_0$.
In particular, when $s=2$, we have
\[
\int_{F^\times\bsl\A^\times}(\hat\tau_Q(H_Q(m))-\hat\tau_{P_1}(H_{P_1}(m)-T))\, \d^\times b=c_F  \, (T_1 - \log|a|)
\]
for any $a$ in $\A^\times$.
\end{lem}

The set of $F$-rational points in $C$ is expressed as
\[
C(F)=\bigsqcup_{\delta\in Q(F)\bsl G(F)}\delta^{-1}\exp(V^0(F))\, U'(F) \, \delta
\]
(see, e.g., \cite[Theorem 5]{Hoffmann}).
Hence, by Lemmas \ref{change}, \ref{tJconv} and \ref{l4} we have
\[
J_C^T(f)=\tJ_C^T(f)=\tZ(f_{\bK,U'},2,T).
\]
It follows from \eqref{prin} and Lemma \ref{l4} that
\begin{align*}
& \tZ(f_{\bK,U'},2,T)-\lim_{s\to 2+0} \frac{d}{\d s}(s-2)Z(f_{\bK,U'},s) \\
&= \int_1^{e^{T_2}} \d^\times t \, \vol_L \, \widehat{f_{\bK,U'}}(0)\\
& \quad  +\frac{\vol_{M_0}}{c_F} \int_{\A^\times}(T_1 - \log|a|) \, |a|^2 \int_{\A^{\oplus 2}} f_{\bK,U'}(0,a,x_3,x_4)\, \d x_3 \, \d x_4 \, \d^\times a .
\end{align*}
We choose the basis $\{X_{\alpha_1},X_{\alpha_1+\alpha_2}\}$ of $V_1(F)$ so that $V_1(F)$ is identified with $F^{\oplus 2}$.
Since $\int_{\A^{\oplus 2}} f_{\bK,U'}(0,a,x_3,x_4)\, \d x_3 \, \d x_4=f_{\bK,N_1'}(0,a)$, it follows from the mean value formula~\eqref{mean} that
\begin{align*}
& \frac{\vol_{M_0}}{c_F} \int_{\A^\times}(T_1 - \log|a|) \, |a|^2 \int_{\A^{\oplus 2}} f_{\bK,U'}(0,a,x_3,x_4)\, \d x_3 \, \d x_4 \, \d^\times a \\
& =\frac{\vol_{M_0}}{c_F} \int_{\A^\times}(T_1 - \log|a|) \, |a|^2 f_{\bK,N_1'}(0,a) \, \d^\times a \\
& =\int_{M_1\cap P_0(F)\bsl M_1(\A)^1}\sum_{y\in F^\times } (T_1 - \log\|\mathrm{Ad}(m)^{-1}(0,y)\|) \,  f_{\bK,N_1'}(\mathrm{Ad}(m)^{-1}(0,y)) \, \d m \\
&=\vol_{M_1}\int_{\A^{\oplus 2} } (T_1 - \log\|(y_1,y_2)\|) \,  f_{\bK,N_1'}(y_1,y_2) \, \d y_1 \, \d y_2.
\end{align*}
Thus, the proof of Theorem \ref{main} is completed.
The above identification $N_1/N_1'(\A)\cong \A^{\oplus 2}$ normalizes the height function in Theorem \ref{main}.

\section{Proofs of two lemmas}
\subsection{Proof of Lemma \ref{tJconv}}\label{prooflem1}

We shall prove that the integral
\begin{multline}\label{bound}
\int_{Q(F)\bsl G(\A)} \Big| e^{-(s-2)\varpi_2(H_Q(g))} \Big\{  \sum_{\mu\in V^0(F)}\sum_{\nu\in U'(F)} f(g^{-1}\exp(\mu)\nu  g)    \\
- \int_{U(\A)} f(g^{-1}u g)\, \d u \, \widehat\tau_Q(H_Q(g)-T)   \\
- \sum_{\delta\in P_0(F)\bsl Q(F)}\sum_{x\in F^\times }  \int_{N_1'(\A)} f(g^{-1}\delta^{-1} \exp(x X_{\alpha_1+\alpha_2}) n' \delta g)\, \d n' \, \widehat\tau_{P_1}(H_{P_1}(\delta g)-T)  \Big\} \, \Big|  \,\d g
\end{multline}
is convergent for $\mathrm{Re}(s)>5/3$.
Lemma \ref{tJconv} follows from the convergence of \eqref{bound} for $s=2$.
We note that
\[
\varpi_2(H_L(l))=\log|\det(l)| \qquad (l\in L(\A)).
\]

To begin with, \eqref{bound} is bounded as follows:
\[
\eqref{bound}\leq \eqref{c1}+\eqref{c2}+\eqref{c3},
\]
\begin{equation}\label{c1}
\int_{L(F)\bsl L(\A)} |\det(l)|^{-\mathrm{Re}(s)} \sum_{\mu\in V^0(F)}|f|_{\bK,U'}(\mu\cdot l)\, \widehat\tau_Q(T-H_Q(l)) \, \d l,
\end{equation}
\begin{equation}\label{c2}
\int_{Q(F)\bsl G(\A)}\Big| e^{-(s-2)\varpi_2(H_Q(g))}\sum_{\mu\in V^0(F)}\sum_{x'\in F^\times} \tilde f(\exp(\mu),x',g)\,  \widehat\tau_Q(H_Q(g)-T)\, \Big| \, \d g ,
\end{equation}
\begin{multline}\label{c3}
\int_{Q(F)\bsl G(\A)} \Big|  e^{-(s-2)\varpi_2(H_Q(g))} \Big\{  \sum_{\mu\in V^0(F)}\tilde f(\exp(\mu),0,g)  \,  \widehat\tau_Q(H_Q(g)-T)  \\
- \int_{U(\A)} f(g^{-1}u g)\, \d u \, \widehat\tau_Q(H_Q(g)-T)   \\
- \sum_{\delta\in P_0(F)\bsl Q(F)}\sum_{x\in F^\times }  \int_{N_1'(\A)} f(g^{-1}\delta^{-1} \exp(x X_{\alpha_1+\alpha_2}) n' \delta g)\, \d n' \, \widehat\tau_{P_1}(H_{P_1}(\delta g)-T)  \Big\} \, \Big| \, \d g,
\end{multline}
where $|f|$ means the function $g\mapsto |f(g)|$ and we set
\[
\tilde f(h,x',g)=\int_{\A}f(g^{-1}h\exp(y'X_{3\alpha_1+2\alpha_2})g)\, \psi_F(x'y') \,  \d y'.
\]
The convergence of \eqref{c1} is obvious for any $s$.
The integral~\eqref{c2} is bounded by
\begin{multline*}
\int_{L(F)\bsl L(\A)}  |\det(l)|^{-\mathrm{Re}(s)} \sum_{\mu\in V^0(F)} \sum_{x'\in F^\times} \\
\int_\bK\, \big|\, \tilde f(\exp(\mu\cdot l),x' \det(l),k)\, \d k\, \widehat\tau_Q(H_Q(l)-T) \, \big|  \, \d l
\end{multline*}
and this is convergent for any $s$ due to the component $x' \det(l)$.
Hence, it is enough to consider \eqref{c3}.
Note that $\tilde f(\exp(\mu),0,g) = \int_{U'(\A)} f(g^{-1}\exp(\mu)u  g)\, \d u$.

For $f\in C_c^\inf(G(\A))$, $x\in V(\A)$, and $g\in Q(F)\bsl G(\A)$, we set
\[
\check f(x,g)=\tilde f(\exp(x),0,g),\quad \hat f(x,g)=\int_{V(\A)}\check f(y,g)\, \psi_F([x,y])  \, \d y.
\]
For the Iwasawa decomposition $g=ulk\in Q(F)\bsl G(\A)$, $(u\in U(F)\bsl U(\A)$, $l\in L(F)\bsl L(\A)$, $k\in\bK)$, it follows that
\[
\check f(x,ulk)= |\det(l)|\times \int_\bK \int_{U'(\A)} f(k^{-1}\exp(x)u'k)\, \d u' \, \d k \quad (x\in V(\A)).
\]
For each fixed $g$ in $Q(F)\bsl G(\A)$, $\check f(x,g)$ and $\hat f(x,g)$ are regarded as Schwartz-Bruhat functions on $V(\A)$.
For any test function $\varphi$ on $V(\A)$, we define $\varphi^{(4)}$ and $\varphi^{(3,4)}$ by
\[
\varphi^{(4)}(x_1,x_2,x_3,x_4)=\int_{\A}\varphi(x_1,x_2,x_3,y_4)\, \psi_F(x_4y_4) \, \d y_4 ,
\]
\[
\varphi^{(3,4)}(x_1,x_2,x_3,x_4)=\int_{\A}\varphi(x_1,x_2,y_3,y_4)\, \psi_F(x_3y_3)\, \psi_F(x_4y_4) \, \d y_3  \, \d y_4 .
\]
The singular set $\{ x\in V \mid P(x)=0 \}$ is decomposed into the three orbits $S_0=\{(0,0,0,0)\}$, $S_1=(0,0,0,1)\cdot L$, and $S_2=(0,0,1,0)\cdot L$.
We write $S_j(F)$ for the set of $F$-rational points of $S_j$.
By the Poisson summation formula, one gets
\begin{multline*}
\sum_{\mu\in V^0(F)}\check f(\mu,g) = -\check f((0,0,0,0),g)  -\sum_{\mu\in S_1(F)}\check f(\mu, g)  -\sum_{\mu\in S_2(F)}\check f(\mu,g) \\
 + \sum_{\mu\in V^0(F)}\hat f(\mu, g) + \int_{U(\A)} f(g^{-1}u g)\, \d u   +\sum_{\mu\in S_1(F)}\hat f(\mu, g) +\sum_{\mu\in S_2(F)}\hat f(\mu, g) .
\end{multline*}
Using the decompositions
\[
S_1(F)=\bigsqcup_{\delta\in B(F)\bsl L(F)}\bigsqcup_{x_4\in F^\times}\{(0,0,0,x_4)\cdot \delta \},
\]
\[
S_2(F)=\bigsqcup_{\delta\in B(F)\bsl L(F)}\bigsqcup_{x_3\in F^\times, \, x_4\in F} \{(0,0,x_3,x_4)\cdot \delta\},
\]
we have the following bound
\[
\eqref{c3}\leq \eqref{e1}+\eqref{e2}+\eqref{e3}   ,
\]
where
\begin{equation}\label{e1}
\int_{Q(F)\bsl G(\A)}\Big| \, e^{-(s-2)\varpi_2(H_Q(g))} \check f((0,0,0,0),g) \widehat\tau_Q(H_Q(g)-T)\, \Big| \, \d g
\end{equation}
is absolutely convergent for $\mathrm{Re}(s)>0$,
\begin{equation}\label{e2}
\int_{Q(F)\bsl G(\A)}\Big|\, e^{-(s-2)\varpi_2(H_Q(g))}\sum_{\mu\in V^0(F)}\hat f(\mu,g) \, \widehat\tau_Q(H_Q(g)-T)\, \Big| \, \d g
\end{equation}
is absolutely convergent for any $s$, and
\begin{multline}\label{e3}
\int_{P_0(F)\bsl G(\A)}\Big|\, e^{-(s-2)\varpi_2(H_Q(g))}\Big\{ -\sum_{x_4\in F^\times}\check f((0,0,0,x_4), g) \,  \widehat\tau_Q(H_Q(g)-T) \\
- \sum_{x_3\in F^\times, \, x_4\in F}\check f((0,0,x_3,x_4),g) \,  \widehat\tau_Q(H_Q(g)-T) \\
+  \sum_{x_4\in F^\times}\hat f((0,0,0,x_4) ,g) \, \widehat\tau_Q(H_Q(g)-T) \\
+ \sum_{x_3\in F^\times, \, x_4\in F}\hat f((0,0,x_3,x_4) ,g) \, \widehat\tau_Q(H_Q(g)-T) \\
- \sum_{x_2\in F^\times }  \int_{A^{\oplus 2}}\check f((0,x_2,y_3,y_4) , g)\, \d y_3 \, \d y_4 \, \widehat\tau_{P_1}(H_{P_1}(g)-T)  \Big\} \, \Big| \, \d g.
\end{multline}
The main difficulty comes from \eqref{e3} for the proof of the convergence of \eqref{bound}.
We use Kogiso's method \cite{Kogiso} to find a convergence range of \eqref{e3}.
Applying the method to \eqref{e3}, we add two dumping terms and divide \eqref{e3} into two integrals as follows:
\[
\eqref{e3}\leq \eqref{e4} + \eqref{e5},
\]
\begin{multline}\label{e4}
\int_{P_0(F)\bsl G(\A)}\Big|\, e^{-(s-2)\varpi_2(H_Q(g))}\,  \widehat\tau_Q(H_Q(g)-T)  \\
\Big\{ -\sum_{x_4\in F^\times}\check f((0,0,0,x_4), g)  +  \sum_{x_4\in F^\times}\hat f((0,0,0,x_4) ,g) \\
-  \hat f^{(4)}((0,0,0,0) ,g)+ \check f^{(4)}((0,0,0,0), g) \Big\} \, \Big| \, \d g,
\end{multline}
\begin{multline}\label{e5}
 \int_{P_0(F)\bsl G(\A)} \Big| \, e^{-(s-2)\varpi_2(H_Q(g))}\Big\{- \sum_{x_3\in F^\times, \, x_4\in F}\check f((0,0,x_3,x_4),g) \,  \widehat\tau_Q(H_Q(g)-T) \\
+ \sum_{x_3\in F^\times, \, x_4\in F}\hat f((0,0,x_3,x_4) ,g) \, \widehat\tau_Q(H_Q(g)-T) \\
+ \hat f^{(4)}((0,0,0,0) ,g) \, \widehat\tau_Q(H_Q(g)-T)- \check f^{(4)}((0,0,0,0), g) \, \widehat\tau_Q(H_Q(g)-T)\\
- \sum_{x_2\in F^\times }  \int_{A^{\oplus 2}}\check f((0,x_2,y_3,y_4) , g)\, \d y_3 \, \d y_4 \, \widehat\tau_{P_1}(H_{P_1}(g)-T)  \Big\}\, \Big| \, \d g.
\end{multline}

To study \eqref{e4} and \eqref{e5}, we need the following notations:
\[
g=nmtk\in G(\A), \quad n\in N_0(F)\bsl N_0(\A) , \;\; m\in M_0(F)\bsl M_0(\A)^1, \;\; k\in \bK, 
\]
\begin{equation}\label{dia}
t=(a^{-1}z,az)  \in A_0^+=\{ (t_1,t_2)\in M_0(\R)\mid t_1,t_2>0 \} \subset M_0(\A) ,
\end{equation}
where an embedding $\R\subset \A$ is chosen so that the absolute value of $r\in\R$ equals the idele norm of $r$.
By using the Poisson summation formula for $x_4$ and dividing the integration domain of $a$ and $z$ into $a^{-3}z>1$ or $a^{-3}z<1$, one gets
\[
\eqref{e4}\leq \eqref{e6}+\eqref{e7},
\]
\begin{multline}\label{e6}
 \int_{\bK}\int_{M_0(F)\bsl M_0(\A)^1} \int_{A_0^+,\; z^2>e^{T_2},\; a^{-3}z>1} z^{-2(\mathrm{Re}(s)-2)} \\
\Big\{ \sum_{x_4\in F^\times} | \check f((0,0,0,a^{-3}z^{-1}x_4), mk)|\, a^{-2}z^{-4} +  \sum_{x_4\in F^\times} | \hat f((0,0,0,a^{-3}zx_4) ,mk)| \, a^{-2} \\
+ | \hat f^{(4)}((0,0,0,0) ,mk)  | \, az^{-1}  + | \check f^{(4)}((0,0,0,0),mk) | \, az^{-3} \Big\}\d^\times a\, \d^\times z\, \d^1 m\, \d k
\end{multline}
is convergent for $\mathrm{Re}(s)>5/3$, and
\begin{multline}\label{e7}
 \int_{\bK}\int_{M_0(F)\bsl M_0(\A)^1} \int_{A_0^+,\; z^2>e^{T_2},\; a^{-3}z<1} z^{-2(\mathrm{Re}(s)-2)} \\
\Big\{ \sum_{x_4\in F^\times} |\check f^{(4)}((0,0,0,a^3zx_4), mk) |\, az^{-3} +  \sum_{x_4\in F^\times} | \hat f^{(4)}((0,0,0,a^3z^{-1}x_4) ,mk)| \, az^{-1} \\
+ | \hat f((0,0,0,0) ,mk)  | \, a^{-2}  + | \check f((0,0,0,0),mk) | \, a^{-2}z^{-4} \Big\}\d^\times a\, \d^\times z\, \d^1 m\, \d k
\end{multline}
is also convergent for $\mathrm{Re}(s)>5/3$.

Next we consider the integral \eqref{e5}.
We again apply the Poisson summation formula to \eqref{e5}.
If the integration domain of $a$ and $z$ are divided into $az<e^{T_1}$ or $az>e^{T_1}$, then we obtain
\[
\eqref{e5}\leq \eqref{e8}+\eqref{e9},
\]
\begin{multline}\label{e8}
 \int_{\bK}\int_{M_0(F)\bsl M_0(\A)^1} \int_{A_0^+,\; z^2>e^{T_2},\; az<e^{T_1}} z^{-2(\mathrm{Re}(s)-2)} \\
\Big\{ \sum_{x_3\in F^\times } \int_\A |\check f((0,0,a^{-1}z^{-1}x_3,x_4),mk)|\, \d x_4 \, az^{-3} \\
+ \sum_{x_3\in F^\times}\int_\A |\hat f((0,0,a^{-1}zx_3,x_4) ,mk)|\, \d x_4 \, az^{-1} \\
+ |\hat f^{(4)}((0,0,0,0) ,mk) | \, az^{-1} + |\check f^{(4)}((0,0,0,0),mk) | \, az^{-3}  \Big\}\d^\times a\, \d^\times z\, \d^1 m\, \d k,
\end{multline}
\begin{multline}\label{e9}
 \int_{P_0(F)\bsl G(\A)} \Big| \, e^{-(s-2)\varpi_2(H_Q(g))}\\
\Big\{- \sum_{x_3\in F^\times, \, x_4\in F^\times}\check f^{(4)}((0,0,x_3,x_4),g) \,  \widehat\tau_0(H_0(g)-T) \\
+ \sum_{x_3\in F^\times, \, x_4\in F^\times}\hat f^{(4)}((0,0,x_3,x_4) ,g) \, \widehat\tau_0(H_0(g)-T) \\
- \sum_{x_3\in F^\times}\check f^{(3,4)}((0,0,x_3,0),g) \,  \widehat\tau_0(H_0(g)-T) \\
+ \sum_{x_3\in F^\times}\hat f^{(3,4)}((0,0,x_3,0) ,g) \, \widehat\tau_0(H_0(g)-T) \\
- \sum_{x_2\in F^\times }  \int_{\A^{\oplus 2}} \check f((0,x_2,y_3,y_4) , g)\, \d y_3 \, \d y_4 \, \widehat\tau_{P_1}(H_{P_1}(g)-T)  \Big\} \, \Big| \, \d g.
\end{multline}
It is obvious that \eqref{e8} is convergent for $\mathrm{Re}(s)>1$.
Notice that, in the above calculation, we used the remarkable equality
\[
\check f^{(3,4)}((0,0,0,0),g)=\hat f^{(3,4)}((0,0,0,0) ,g).
\]
Since
\[
\sum_{x_3\in F^\times}\hat f^{(3,4)}((0,0,x_3,0) ,g)=\sum_{x_2\in F^\times }  \int_{A^{\oplus 2}} \check f((0,x_2,y_3,y_4) , g)\, \d y_3 \, \d y_4,
\]
one finds
\[
\eqref{e9}\leq \eqref{e12}+\eqref{e13},
\]
\begin{multline}\label{e12}
 \int_{\bK}\int_{M_0(F)\bsl M_0(\A)^1} \int_{A_0^+,\; z^2>e^{T_2},\; az>e^{T_1}} z^{-2(\mathrm{Re}(s)-2)} \\
\Big\{ \sum_{x_3\in F^\times, \, x_4\in F^\times} |\check f^{(4)}((0,0,a^{-1}z^{-1}x_3,a^3 zx_4),mk)|\, az^{-3}  \\
+ \sum_{x_3\in F^\times, \, x_4\in F^\times} |\hat f^{(4)}((0,0,a^{-1}zx_3, a^3z^{-1}x_4) ,mk)|\, az^{-1}   \\
+ \sum_{x_3\in F^\times}|\check f^{(3,4)}((0,0,azx_3,0),mk)| \, a^2z^{-2}  \Big\}\d^\times a\, \d^\times z\, \d^1 m\, \d k,
\end{multline}
\begin{multline}\label{e13}
\int_{M_0(F)\bsl M_0(\A)^1} \int_{A_0^+,\; z^2<e^{T_2},\; az>e^{T_1}}  z^{-2(\mathrm{Re}(s)-2)}\\
 \sum_{x_2\in F^\times }  \int_{\A^{\oplus 2}} \big| \check f((0,az^{-1}x_2,y_3,y_4) ,mk) \, \big| \, a^2z^{-2} \, \d y_3 \, \d y_4 \, \d^\times a\, \d^\times z\, \d^1 m\, \d k .
\end{multline}
\eqref{e12} comes from the first three terms in \eqref{e9} and it is clear that \eqref{e12} converges for $\mathrm{Re}(s)>1$.
The remaining part of \eqref{e9} is bounded by \eqref{e13}.
Since $az^{-1}>e^{T_1-T_2}$ and $(az^{-1})^{-1/2}e^{T_1/2}<z<e^{T_2/2}$ hold over the integration domain, one finds that \eqref{e13} is absolutely convergent for any $s$.
Hence, \eqref{e9} is convergent for $\mathrm{Re}(s)>1$. 
Thus, the proof is completed.

\subsection{Proof of Lemma \ref{l4}}\label{prooflem2}

It follows from the proof of Lemma \ref{tJconv} that $\tZ(\varphi,s,T)$ absolutely converges for $\mathrm{Re}(s)>5/3$.
Hence, we will prove the latter part of Lemma \ref{l4}.
We may assume that $\varphi(x\cdot k)=\varphi(x)$ holds for any $x\in V(\A)$ and any $k\in\bK\cap L(\A)$ without loss of generality.
Furthermore, we set
\[
(l\cdot \varphi)(x)=\varphi(x\cdot l) \qquad (x\in V(\A),\;\; l\in L(\A))
\]
and the notation $\varphi^{(4)}$, $\hat\varphi^{(3,4)}$ etc. is defined in the same manner as in Section~\ref{prooflem1}.
Note that $(\widehat{l\cdot\varphi})(x)=\hat\varphi(x\cdot l^\iota)\times |\det(l)|^2$ holds, where $l^\iota=\det(l)^{-1}l$ $(l\in L)$.

By the Poisson summation formula, for $\mathrm{Re}(s)>5/3$ we have
\begin{align*}
\tZ(\varphi,s,T)= & Z_+(\varphi,s)+Z_+(\hat\varphi,2-s) - \frac{1}{s} \, \vol_L \, \varphi(0) + \int_1^{e^{T_2}} t^{-(s-2)}\, \d^\times t \, \vol_L \, \hat\varphi(0) \\
& + I_1(\varphi,s)+I_2(\varphi,s,T),
\end{align*}
\begin{align*}
I_1(\varphi,s)=& \int_{B(F)\bsl L(\A)} |\det(l)|^{-s} \Big\{ -\sum_{x_4\in F^\times}(l\cdot \varphi)(0,0,0,x_4) + \sum_{x_4\in F^\times}(\widehat{l\cdot \varphi})(0,0,0,x_4) \\
&- (\widehat{l\cdot \varphi})^{(4)}(0,0,0,0)+(l\cdot \varphi)^{(4)}(0,0,0,0) \Big\} \, \widehat\tau_Q(H_L(l)) \, \d l,
\end{align*}
\begin{align*}
I_2(\varphi,s,T)=& \int_{B(F)\bsl L(\A)} |\det(l)|^{-s} \Big\{ -\sum_{x_3\in F^\times,\, x_4\in F}(l\cdot \varphi)(0,0,x_3,x_4) \, \widehat\tau_Q(H_L(l)) \\
& + \sum_{x_3\in F^\times,\, x_4\in F}(\widehat{l\cdot \varphi})(0,0,x_3,x_4) \, \widehat\tau_Q(H_L(l)) \\
&+ (\widehat{l\cdot \varphi})^{(4)}(0,0,0,0) \, \widehat\tau_Q(H_L(l))-(l\cdot \varphi)^{(4)}(0,0,0,0) \, \widehat\tau_Q(H_L(l)) \\
& - \sum_{x\in F^\times } \, \int_{\A^{\oplus 2}} (l\cdot \varphi)(0,x,x_3,x_4) \, \d x_3 \, \d x_4  \, \widehat\tau_{P_1}(H_{P_1}(l)-T)  \Big\} \d l. 
\end{align*}
It follows from the proof of Lemma \ref{tJconv} (namely, the convergence of \eqref{e4}) that $I_1(\varphi,s)$ is absolutely convergent for $\mathrm{Re}(s)>5/3$.
Using the argument in \cite[Proof of Proposition 2.5, p.242--245]{Kogiso} and the meromorphic continuation of $\Sigma_1(\varphi,s)$, one can show the equality
\[
I_1(\varphi,s)=- \frac{\vol_{M_0}}{c_F\, (3s-1)}\, \Sigma_1(\varphi,2/3) + \frac{\vol_{M_0}}{c_F\, (3s-5)}\, \Sigma_1(\hat\varphi,2/3).
\]

The integral $I_2(\varphi,s,T)$ is absolutely convergent for $\mathrm{Re}(s)>1$ by the proof of Lemma \ref{tJconv} (namely, the convergence of \eqref{e5}).
Let $N_B=N_0\cap L$ and set
\[
l=nmk\in B(F)\bsl L(\A), \;\; n\in N_B(F)\bsl N_B(\A) , \;\; m\in M_0(F)\bsl M_0(\A), \;\; k\in \bK\cap L(\A) .
\]
Then, one finds
\begin{multline*}
I_2(\varphi,s,T)= \int_{M_0(F)\bsl M_0(\A)}  |\det(m)|^{-s}\delta_B(m)^{-1} \\
 \Big\{ \Big(- \sum_{x_3\in F}(m\cdot \varphi)^{(4)}(0,0,x_3,0)
+ \sum_{x_3\in F}(\widehat{m\cdot \varphi})^{(4)}(0,0,x_3,0)\Big) \widehat\tau_Q(H_L(m)) \\
- \sum_{x_2\in F^\times}\int_{\A^{\oplus 2}}(m\cdot \varphi)(0,x_2,x_3,x_4)\, \d x_3 \, \d x_4 \,\widehat\tau_{P_1}(H_{P_1}(m)-T)\Big\} \d m .
\end{multline*}
Since $(m\cdot \varphi)^{(3,4)}(0,0,0,0) = (\widehat{m\cdot \varphi})^{(3,4)}(0,0,0,0)$, using the  Poisson summation formula, we get
\begin{multline*}
I_2(\varphi,s,T)= \int_{M_0(F)\bsl M_0(\A)}  \,|\det(m)|^{-s}\delta_B(m)^{-1} \\
 \Big\{ \Big(- \sum_{x_3\in F^\times}(m\cdot \varphi)^{(3,4)}(0,0,x_3,0)
+ \sum_{x_3\in F^\times }(\widehat{m\cdot \varphi})^{(3,4)}(0,0,x_3,0) \Big) \widehat\tau_Q(H_L(m)) \\
- \sum_{x_2\in F^\times}\int_{\A^{\oplus 2}}(m\cdot \varphi)(0,x_2,x_3,x_4)\, \d x_3 \, \d x_4 \,\widehat\tau_{P_1}(H_{P_1}(m)-T)\Big\}\, \d m.
\end{multline*}
Finally, with the notation $m=(a,b)$ we derive
\begin{multline*}
I_2(\varphi,s,T)=- \frac{\vol_{M_0}}{c_F \, s} \int_{\A^\times}\int_{\A^{\oplus 2}} |x|^2 \hat\varphi(0,x,x_3,x_4) \, \d x_3\, \d x_4 \, \d^\times x \\
+\frac{\vol_{M_0}}{c_F^2} \, \int_{F^\times\bsl\A^\times} \int_{\A^\times}
\int_{\A^{\oplus 2}}  \varphi(0,a^{-1},x_3,x_4) \, \d x_3 \, \d x_4 \\
    \quad \times |a|^{-s} |b|^{2-s}
(\hat\tau_Q(H_Q(m))-\hat\tau_{P_1}(H_{P_1}(m)-T))\,\d^\times a\, \d^\times b   
\end{multline*}
from the facts
\[
(m\cdot \varphi)^{(3,4)}(0,0,x,0)=\int_{\A^{\oplus 2}}(\widehat{m\cdot \varphi})(0,3x,x_3,x_4) \, \d x_3 \, \d x_4,
\]
\[
(\widehat{m\cdot \varphi})^{(3,4)}(0,0,x,0)=\int_{\A^{\oplus 2}}(m\cdot \varphi)(0,3x,x_3,x_4) \, \d x_3 \, \d x_4.
\]
Hence, the proof of Lemma \ref{l4} is completed.

\end{document}